\documentclass[11pt]{article} 
\usepackage{amssymb, latexsym}
\usepackage{import}
\usepackage{dcpic, pictexwd, graphicx}
\usepackage[dvips]{color}
\usepackage{epsfig, array}
\newtheorem{defin}{}
\newtheorem{saetze}[defin]{}
\newtheorem{conjec}[defin]{}
\newtheorem{lemmas}[defin]{}
\newtheorem{folger}[defin]{}
\newtheorem{bemerk}[defin]{}

\newenvironment{defi}     {\begin{defin}\it {\bf Definition:}}{\end{defin}}
\newenvironment{theorem}  {\begin{saetze}\it {\bf Theorem:}}{\end{saetze}}

\newenvironment{lemma}    {\begin{lemmas}\it {\bf Lemma:}}{\end{lemmas}}

\newenvironment{remark}   {\begin{bemerk}\it {\bf Remark:}}{\end{bemerk}}
\newenvironment{proof}    {{\it Proof}:}{{\hfill \fillbox \bigskip}}

\newcommand{\fillbox}{\mbox{$\bullet$}}
\newcommand{\ra}{\rightarrow}

\newcommand{\ms}{\mapsto}
\newcommand{\ol}{\overline}

\newcommand{\split}{\rtimes}
\newcommand{\N}{\mathbb N}

\newcommand{\Z}{\mathbb Z}

\newcommand{\G}{\mathcal G}
\newcommand{\T}{\mathcal T}

\newcommand{\W}{\mathcal W}
\renewcommand{\P}{\mathcal P}
\renewcommand{\v}{\ol{v}}
\newcommand{\w}{\ol{w}}
\renewcommand{\o}{\ol{o}}

\newcolumntype{C}[1]{>{\centering $}p{#1}<{$}} 
\newcolumntype{L}[1]{>{$}p{#1}<{$}} 
\newcolumntype{R}[1]{>{\raggedleft\arraybackslash $}p{#1}<{$}}

\newenvironment{items}{\begin{list}{$\alph{item})$}
{\labelwidth18pt \leftmargin18pt \topsep3pt \itemsep1pt \parsep0pt}}
{\end{list}}

\newcommand{\bulit}{\item[$\bullet$]}

\begin{document}

\title{Infinite sequences of $p$-groups with fixed coclass}
\author{Bettina Eick and D\"orte Feichtenschlager}
\date{\today}
\maketitle

\begin{abstract}
Eick \& Leedham-Green sketched a construction for infinite sequences
of finite $p$-groups with fixed coclass. These infinite sequences have
turned out to be very useful in the theory of finite $p$-groups. We 
exhibit a detailed description for the construction of the finite 
sequences and we determine presentations for the infinite sequences for 
the primes and coclasses $(2,1)$, $(2,2)$ and $(3,1)$.
\end{abstract}

\section{Introduction}

Coclass theory has been initiated by Leedham-Green and Newman \cite{LNe80}:
they suggested to classify and investigate finite $p$-groups using their 
coclass as primary invariant. Recall that the coclass of a group of order
$p^n$ and class $c$ is defined as $n-c$. The suggestion to classify $p$-groups
by coclass has led to a rich and interesting research project. The book by 
Leedham-Green and McKay \cite{LGM02} describes its state of the art up until 
2002.

The graphs $\G(p,r)$ visualize the finite $p$-groups of coclass $r$ and they
play a central role in coclass theory. A recent discovery is the existence of 
periodic patterns in the graphs. These have been conjectured by Newman 
\& O'Brien \cite{NOB99} and a first proof has been given by du Sautoy 
\cite{DuS01}. Eick \& Leedham-Green \cite{ELG08} exhibited a new proof which
additionally yields a group theoretic construction for infinite sequences of 
finite $p$-groups with coclass $r$ associated to the periodic patterns in 
$\G(p,r)$.

The infinite sequences have proved to be useful in the theory of finite 
$p$-groups. For example, \cite{Eic06b} determines a formula for the orders
of the automorphism groups of the groups in an infinite sequences of 
$2$-groups. The result implies that among the finite $2$-groups of coclass 
$r$ there are at most finitely many counterexamples to the long-standing 
divisibility conjecture: `if $G$ is a non-abelian finite $p$-group, then 
$|G|$ divides $|Aut(G)|$'. Further, \cite{Eic07b, EFe10b} investigate the 
Schur multiplicators of the groups in an infinite sequence. \cite{EFe10b} 
shows that these Schur multiplicators can be described by a single 
parametrised presentation. It is conjectured that this type of behaviour 
extends to homology groups and cohomology rings in general; See also 
\cite{DEF08} for some details.
 
The explicit definition of the infinite sequences is a side-result in 
\cite{ELG08} and it is not easily extracted from that paper. Further, in 
many of the applications of the infinite sequences minor variations on 
their construction are needed. It is the aim of this note to provide a 
detailed description for the construction of the infinite sequences of
finite $p$-groups of coclass $r$ so that a range of minor variations
is available as well. As example, we construct the infinite sequences 
for the graphs $\G(2,1)$, $\G(3,1)$ and $\G(2,2)$.

\section{Preliminaries}

In this section we summarize some of the well known major features of
coclass theory. For background and proofs we refer to the book by 
Leedham-Green and McKay \cite{LGM02}.
Throughout, for any group $G$ we denote with $G = \gamma_1(G) \geq
\gamma_2(G) \geq \ldots$ the lower central series of $G$ and we write
$G_{/i} = G / \gamma_i(G)$ for its quotients.

\subsection{Infinite pro-$p$-groups of finite coclass}

Let $S$ be a pro-$p$-group with finite commutator quotient. Then every
lower central series quotient $S_{/i}$ is a finite $p$-group. We use this
to define the coclass of $S$ as $lim_{i \ra \infty} cc(S_{/i})$. Thus if
$S$ is a pro-$p$-group of coclass $r$, then there exists an integer $u_1$ 
such that $cc(S_{/i}) = r$ and $[\gamma_i(S):\gamma_{i+1}(S)] = p$ holds for 
every $i \geq u_1$. 

A general structure theory for the infinite pro-$p$-groups of coclass $r$ 
is available. A deep result of this theory asserts that
every such group $S$ is solvable. More precisely, every such group $S$ has 
a torsion-free abelian normal subgroup of finite index. This implies that 
there exists an integer $u_2$ such that $\gamma_{u_2}(S)$ is torsion-free
abelian and hence a direct product of finitely many copies of the $p$-adic 
integers $\Z_p$. The number of copies $d$ is an invariant for $S$, called 
its {\em dimension}. 

A fundamental theorem from coclass theory shows that there are only
finitely many isomorphism types of infinite pro-$p$-groups of coclass $r$. 
Further, every pair $S$ and $\ol{S}$ of non-isomorphic infinite 
pro-$p$-groups of coclass $r$ can be distinguished by finite quotients: 
there exists an integer $u_3$ such that $S_{/u_3} \not \cong \ol{S}_{/u_3}$
holds.

We define the {\em primary root} of an infinite pro-$p$-group $S$ of
coclass $r$ and dimension $d$ as the smallest integer $u$ such that the
quotient $S_{/u}$ is a finite $p$-group of coclass $r$, the subgroup 
$\gamma_u(S)$ is torsion-free abelian and $S_{/u}$ is not isomorphic
to $\ol{S}_{/u}$ for any infinite pro-$p$-group $\ol{S}$ of coclass $r$ 
with $\ol{S} \not \cong S$.

\subsection{Coclass graphs}

The {\em coclass graph} $\G(p,r)$ is a tool which is used to visualize 
the $p$-groups of coclass $r$. Its vertices correspond one-to-one to the 
isomorphism types of finite $p$-groups of coclass $r$. Two vertices $G$ 
and $H$ are connected by an edge (sometimes directed $G \ra H$) if 
$H/\gamma(H) \cong G$, where $\gamma(H)$ is the last non-trivial term of 
the lower central series of $H$. 

The {\em infinite paths} in $\G(p,r)$ correspond to the infinite 
pro-$p$-groups of coclass $r$. Every infinite path in $\G(p,r)$ defines an 
infinite pro-$p$-group of coclass $r$ via the inverse limit of the groups 
on the path. Conversely, every infinite pro-$p$-group of coclass $r$ 
yields an infinite path in $\G(p,r)$ via its lower central series 
quotients of coclass $r$. 

The {\em coclass tree} $\T(S)$ in $\G(p,r)$ associated with an infinite 
pro-$p$-group $S$ of coclass $r$ is the subtree of $\G(p,r)$ which contains 
all descendants of the quotient $S_{/u}$ where $u$ is the primary root of $S$.
The groups $S_{/u}, S_{/u+1}, \ldots$ define an infinite path through $\T(S)$.
The definition of primary root implies that this is the only infinite path 
in $\T(S)$ starting at the root of this tree.

The {\em depth} of a group $G$ in a coclass tree $\T(S)$ is its distance
to the infinite path in $\T(S)$. The {\em shaved} tree $\T_k(S)$ is the 
subtree of $\T(S)$ containing all groups of depth at most $k$. The {\em 
depth} of $\T(S)$ (or $\T_k(S)$) is the maximal depth of the groups in the
tree. The depth of $\T_k(S)$ is bounded by $k$, while the depth of $\T(S)$
can be finite or infinite. Note that the depth of every coclass tree
in $\G(2,r)$ and in $\G(3,1)$ is known to be finite, while every other
graph $\G(p,r)$ contains coclass trees of unbounded depth.

As there are only finitely many isomorphism types of infinite pro-$p$-groups 
of coclass $r$, it follows that $\G(p,r)$ contains only finitely many 
different coclass trees. Further, it is not difficult to show that there 
are only finite many groups in $\G(p,r)$ which are not contained in any
coclass tree. 

\section{Graph periodicity}
\label{period}

The periodicity of the shaved coclass trees is proved in \cite{DuS01} and
\cite{ELG08}. It asserts that for each integer $k$ and each infinite 
pro-$p$-group $S$ of finite coclass the shaved coclass tree $\T_k(S)$ 
is virtually periodic. For a more precise description, let $d$ be the 
dimension of $S$ and let $\T_{k,j}(S)$ be the subtree of all descendants 
of $S_{/j}$ in $\T_k(S)$. Then there exists a $j \in \N$ such that there
exists a graph isomorphism 
\[ \pi : \T_{k,j}(S) \ms \T_{k,j+d}(S).\]

This graph isomorphism allows to define infinite sequences of finite 
$p$-groups of coclass $r$ in $\T_k(S)$. Consider all the (finitely many) 
groups $G$ in $\T_{k,j}(S) \setminus \T_{k,j+d}(S)$. For each such group 
$G$, define an infinite sequence $(G_0, G_1, \ldots)$ via taking $G_0=G$ 
and $G_{i+1}$ as the image of $G_i$ under the graph isomorphism $\pi$.

This result asserts that $\T_k(S)$ consists of finitely many
infinite sequences (corresponding one-to-one to the groups in $\T_{k,j}(S)
\setminus \T_{k,j+d}(S)$) and the finitely many other groups (contained
in $\T(S) \setminus \T_{k,j}(S)$). However, so far the result is
practically useless in group theoretic applications, as it is based on a 
purely graph theoretic isomorphism. 

In the remainder of this paper we outline an alternative group theoretic 
definition for the infinite sequences in $\T_k(S)$. 

\section{A first overview}

Suppose we are given an infinite pro-$p$-group $S$ with coclass $r$, 
dimension $d$ and primary root $u$ and an integer $k$. This section provides
a first overview on the construction of infinite sequences in the shaved 
coclass tree $\T_k(S)$. Following the ideas of \cite{ELG08}, we proceed in 
two steps.

First, we choose a set $\P = \P(S,k)$ consisting of certain tuples of 
integers of the form $(l,e)$ with $l = l(S,k)$ and $e = e(S,l)$. A necessary 
condition for all tuples is that $l \geq u$ holds. The entries $l$ are 
called {\em secondary roots} and the entries $e$ are called {\em offsets}.

Then, we consider each pair $(l,e)$ in $\P$ in turn and determine all 
infinite sequences in $\T_k(S)$ associated to it. For this purpose, let 
$A_i = A_i(l,e) := \gamma_l(S)/ \gamma_{l+e+id}(S)$ and note that $A_i$ is 
finite abelian of order $p^{e+id}$, as $l \geq u$. The natural conjugation 
action of $S$ on $A_i$ allows to consider $A_i$ as an
$S_{/l}$-module. We now construct infinite 
sequences of groups $(G_i \mid i \in \N_0)$ using extension theory: the 
group $G_i$ is an extension of $A_i$ by $S_{/l}$ via a certain special cocycle 
$\delta_i$. 

The following questions remain to be discussed.
\medskip

{\bf Questions:}{\it 
\begin{items}
\item[\rm (1)]
How do we choose a suitable set $\P$?
\item[\rm (2)]
How do we choose the cocycles $\delta_i$ for the extensions?
\item[\rm (3)]
Why does that yield a complete set of infinite sequences and is
that set irredundant?
\end{items}}
\medskip

Before we discuss these questions in the subsequent sections of this paper, 
we exhibit some elementary comments on the construction. 

\begin{remark}
\label{orders}
The orders of the groups in an infinite sequence $(G_i \mid i \in \N_0)$
associated with the infinite pro-$p$-group $S$ and the pair $(l,e)$ can 
be read off as
\[ |G_0| = |S_{/l}| p^e = p^{r+1+l+e} \;\; \mbox{ and } \;\;
   |G_{i+1}| = p^d |G_i| \;\; \mbox{ for } \;\; i \geq 0. \]
\end{remark}

Further, it is not difficult to exhibit not only the order, but also the 
isomorphism 
types of the groups $A_i$. The following lemma shows that these are always 
nearly homocyclic and they are homocyclic if and only if the offset $e$ is 
divisible by the dimension $d$. 

\begin{lemma}
\label{isomtype}
Let $S$ be an infinite pro-$p$-group of coclass $r$, dimension $d$ and
primary root $u$. Let $l \geq u$ and write $e = q d + t$ for some $q,t
\in \N_0$ with $0 \leq t < d$. Then the group $A_i = \gamma_l(S)/ 
\gamma_{l+e+id}(S)$ has the abelian invariants $(a_1, \ldots, a_t, a_{t+1}, 
\ldots, a_d)$ with $a_j = p^{i+q+1}$ for $1 \leq j \leq t$ and $a_j = p^{i+q}$ 
for $t < j \leq d$.
\end{lemma}

\begin{proof}
The definition of primary root implies that $[\gamma_j(S):\gamma_{j+1}(S)]
= p$ for all $j \geq l$. This yields that the groups $\gamma_j(S)$ are the 
only $S$-normal subgroups in $\gamma_l(S)$. As $\gamma_l(S) \cong \Z_p^d$,
it follows that $\gamma_l(S)^{p^j} = \gamma_{l+jd}(S)$. This yields the
desired result.
\end{proof}

The choice of the cocycles will induce that all groups in an infinite 
sequence $(G_i \mid i \in \N_0)$ have the same depth in $\T_k(S)$. The
following picture sketches their layout.

\begin{figure}[htbp]
\begin{center}
\input{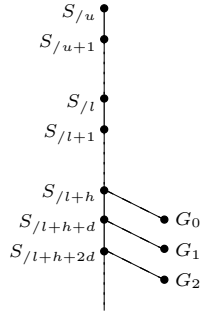}
\caption{Sketch of infinite sequence in $\T_k(S)$ with depth $e-h \le k$.}
\end{center}
\end{figure}

\section{Choosing a set $\P(S,k)$}
\label{secpairs}

In this section we assume that we are given an infinite pro-$p$-group $S$
of coclass $r$ and an integer $k$. Let $d$ be the dimension of $S$ and let 
$u$ be its primary root. Our aim is to determine a set $\P(S,k)$ such 
that every infinite sequence in $\T_k(S)$ can be constructed from exactly 
one entry in this set. 

\begin{defi}
An integer $l$ is a {\em secondary root} for $S$ and $k$ if $l \geq u$ holds
and, secondly, if every descendant $G$ of $S_{/l}$ in $\T_k(S)$ is isomorphic 
to an extension of $\gamma_l(S)/ \gamma_{cl(G)+1}(S)$ by $S_{/l}$. 
\end{defi}

If a group $G$ is a descendant of $S_{/l}$ in $\T_k(S)$, then $G_{/l} 
\cong S_{/l}$ follows. This allows to consider $\gamma_l(G)$ as module
for $S_{/l}$ via the natural conjugation action of $G_{/l}$. The second 
condition in Definition \ref{secroot} is thus equivalent to the following: 
every group $G$ in $\T_{k,l}(S)$ satisfies that $\gamma_l(G)$ is isomorphic 
as $S_{/l}$-module to $\gamma_l(S)/ \gamma_{cl(G)+1}(S)$.

It is not obvious that a secondary root always exists. This is a deep 
result of coclass theory and we refer to Chapters 5 and 11 of \cite{LGM02}
for background. We note that the results there also imply the following
remark.

\begin{remark}
Let $S$ be an infinite pro-$p$-group of finite coclass and let $k \in \N$.
If $l$ is a secondary root for $S$ and $k$, then every $g \geq l$ is also 
one.
\end{remark}

While a primary root $u$ is usually not difficult to determine, it is often 
not straightforward to find an integer which satisfies the second condition, 
as this requires detailed knowledge about almost all groups in the graph 
$\T_k(S)$. This problem is addressed in the following theorem which 
follows directly from Theorem 15 in \cite{ELG08}.

\begin{theorem}
\label{secroot}
Let $S$ be an infinite pro-$p$-group of coclass $r$ and dimension $d$.
Let $u$ be the primary root of $S$ and let $k \in \N$. Then every integer
$l$ with $l \ge u-1 + \max \{p^r, \frac{3}{2} (k+d)\}$ is a secondary root 
for $S$ and $k$.
\end{theorem}

Suppose now that a secondary root $l$ for $S$ and $k$ is given. To shorten
notation let $R = S_{/l}$ and $T = \gamma_l(S)$. Let $H$ be the kernel of the 
action of $R$ on $T$ with commutator quotient $H/H'$ and Schur multiplicator 
$M(H)$. Define the integers $a$ and $b$ by $p^a = exp(H/H')$ and $p^b = 
exp(M(H))$. 

\begin{defi}
An integer $e$ is an {\em offset} for $S$ and $l$ if
\[ e \geq \max \{ 2d(a+b+1), d(l+r-1)\}. \]
\end{defi}

The following theorem simplifies the determination of offsets. It follows 
directly from Theorem 29 in \cite{ELG08}.

\begin{theorem}
\label{offset}
Let $S$ be an infinite pro-$p$-group of coclass $r$ and let $l$ be a
secondary root for $S$. Then every $e \geq 2d(2l + 2r - 1)$ is an 
offset for $S$ and $l$.
\end{theorem}

It remains to discuss how a complete and irredundant set of secondary 
roots and associated offsets can be obtained. 

\begin{defi}
\label{pairs}
We choose $\P(S,k) = \{(l_1,e_1), \ldots, (l_d, e_d)\}$ such that for 
$1\leq i \leq d$ the following conditions hold:
\begin{items}
\bulit $l_i$ is a secondary root for $S$ and $k$,
\bulit $e_i$ is an offset for $S$ and $l_i$, and
\bulit $l_i+e_i \equiv i \bmod d$ for $1 \leq i \leq d$.
\end{items}
\end{defi}

This provides a set of secondary roots and associated offsets. Its 
completeness is proved in Section \ref{complete} below. Its 
irredundancy can be observed easily, as the orders of the groups in 
an infinite sequence associated with $(l,e)$ depend on the values of 
$l+e \bmod d$ by Remark \ref{orders}. Hence every infinite sequence can 
be associated with at most one pair $(l,e)$ in $\P(S,k)$.

\section{Constructing infinite sequences}
\label{secinfseq}

In this section we assume that we have given an infinite pro-$p$-group 
$S$, an integer $k$, a secondary root $l$ and an offset $e$. We show how 
to construct all infinite sequences in $\T_k(S)$ associated with $(l,e)$.

Let $T = \gamma_l(S)$ and $R = S/T$. The choice for $l$ implies that $R$ 
is a finite $p$-group of coclass $r$ and $T$ is isomorphic to $\Z_p^d$.
Denote $T_i = \gamma_{l+i}(S)$ and $A_i = T/T_{e+id}$ for $i \in \N$. The 
conjugation action of $S$ on $T$ induces that $A_i$ is an $R$-module.

The groups in an infinite sequence $\G = (G_i \mid i \in \N_0)$ are defined 
as extensions of $A_i$ by $R$ via certain cocycles $\delta_i \in H^2(R,A_i)$. 
We describe a construction for these cocycles below. We consider the groups 
on the infinite path of $\T_k(S)$ separately as a first step, since these 
are a particularly easy case.

\subsection{Sequences on the infinite path}
\label{infpath}

The projection $T \ra A_i$ induces a natural map $\nu_i : H^2(R,T) \ra 
H^2(R, A_i)$. Let $\alpha \in H^2(R,T)$ be an arbitrary, but fixed element 
which defines the infinite pro-$p$-group $S$ as extension of $T$ by $R$ and 
let $\alpha_i$ denote its image under the projection map. Then $\alpha_i$ 
defines $S_{/l+e+id}$ as extension of $A_i$ by $R$.

\begin{defi}
The infinite sequence defined by $0$ with respect to $S$ and $(l, e)$ 
is the sequence of groups $(G_i \mid i \in \N_0)$ with $G_i = S_{/l+e+id}$. 
\end{defi}

Note that this is an infinite sequence of groups on the infinite path of
$\T_k(S)$. There are $d$ different infinite sequences of groups on the
infinite path. These arise with the construction described here by
varying the values for $(l,e)$.

\subsection{Arbitrary infinite sequences}

We now consider the construction of arbitrary infinite sequences with 
secondary root $l$ and offset $e$. This mainly requires the definition
of suitable cocycles in $H^2(R, A_i)$. We first investigate the structure
of this cohomology group. As above, let $H$ be the kernel of the action 
of $R$ on $T$ and let $p^a = exp(H/H')$. For $i \in \N_0$ define $B_i = 
T_{d(a+1)+id}/T_{e+id}$ and note that $B_i$ is an $R$-invariant subgroup 
of $A_i$. Further, note that Lemma \ref{isomtype} implies that $A_0 
\cong A_i^{p^i}$ so that we can identify $A_0$ with an $R$-invariant
subgroup of $A_i$. We define
\begin{eqnarray*}
\pi_i &:& H^2(R, A_0) \ra H^2(R, A_i) 
            \mbox{ induced by inclusion } A_0 \cong A_i^{p^i} \ra A_i, \\
\mu_i &:& H^2(R, B_i) \ra H^2(R, A_i) 
            \mbox{ induced by inclusion } B_i \ra A_i, \\ 
\nu_i &:& H^2(R, T) \ra H^2(R, A_i) 
            \mbox{ induced by projection } T \ra A_i. 
\end{eqnarray*}

Denote with $N_i$ and $M_i$ the images of $\nu_i$ and $\mu_i$, respectively.

\begin{theorem}
\label{isom}
Let $i \in \N_0$.
\begin{items}
\item[\rm a)]
$H^2(R, A_i) = N_i \oplus M_i$.
\item[\rm b)]
$N_i \cong H^2(R,T)$ and $M_i \cong H^3(R, T_e)$.
\item[\rm c)]
$\pi_i$ restricts to an isomorphism from $M_0$ onto $M_i$.
\end{items}
\end{theorem}

\begin{proof}
Theorem 18 of \cite{ELG08} and the definition of offsets yields that
$H^2(R, A_i) \cong H^2(R,T) \oplus H^3(R, T_{e+id})$. Powering by $p^i$ 
induces an isomorphism $T_e \cong T_{e+id}$ and thus an isomorphism 
$H^3(R,T_e) \cong H^3(R, T_{e+id})$. The proof of Theorem 18 in 
\cite{ELG08} implies that the direct component $H^2(R,T)$ corresponds
to the subgroup $N_i$ of $H^2(R,A_i)$. Theorem 19 in \cite{ELG08} shows
that the direct component $H^3(R, T_{e+id})$ corresponds to the subgroup
$M_i$ of $H^2(R, A_i)$ and that $\pi_i$ is an isomorphism from $M_0$
onto $M_i$. 
\end{proof}

Recall that $\alpha \in H^2(R,T)$ defines $S$ as extension of $T$ by $R$
and that $\alpha_i$ is the image of $\alpha$ under $\nu_i$. We often 
identify $H^3(R,T_e)$ with $M_0$ in the following using the isomorphism
obtained in Theorem \ref{isom} b).

\begin{defi}
The infinite sequence defined by $\beta \in H^3(R,T_e) \cong M_0$ with 
respect to $S$ and $(l,e)$ is the sequence of groups $(G_i \mid i \in 
\N_0)$ where $G_i$ is the extension of $A_i$ by $R$ via the cocycle class
$\delta_i = \alpha_i + \pi_i(\beta)$.
\end{defi}

As $R$ is a finite group, it follows that $M_0 \cong H^3(R,T)$ is a finite 
group and thus this construction yields finitely many infinite sequences. We
note that different cocycle classes $\beta$ and $\beta'$ can yield
infinite sequences whose groups are pairwise isomorphic. We will not
investigate this problem further and instead refer to \cite{ELG08} for
a full solution of the isomorphism problem. We conclude this section with
the following remark.

\begin{remark}
\label{sum}
$\G_\beta$ is the infinite sequences on the infinite path if and only 
if $\beta = 0$.
\end{remark}

\section{Completeness and redundancy}
\label{complete}

In the first part of this section, we show that the construction introduced 
above yields a complete set of infinite sequences in a shaved coclass tree 
$\T_k(S)$.

\begin{theorem}
Let $S$ be an infinite pro-$p$-group of coclass $r$ and $k$ an integer.
Let $\P(S,k)$ be a set of pairs satisfying the conditions of Definition
\ref{pairs}. Then almost all groups in $\T_k(S)$ are contained in an 
infinite sequence.
\end{theorem}

\begin{proof}
Our proof relies heavily on the results of \cite{ELG08}. Let $l_0 = \max
\{ l \mid (l,e) \in \P \}$ and $e_0 = \max \{e \mid (l,e) \in \P\}$. Let 
$G$ be a group in $\T_k(S)$ which is a descendant of $S_{/l_0+e_0}$ and 
has order at least $p^{l_0 + 2e_0 + r-1}$. Note that almost all groups in 
$\T_k(S)$ have this form. We show that $G$ is contained in an infinite 
sequence.

Let $|G| = p^{r-1+x}$ for some $x \geq l_0+2e_0$. Let $(l,e) \in \P$ with 
$l+e \equiv x \bmod d$. Then $G$ is a descendant of $S_{/l}$ as well, as 
$l \leq l_0$. Let $R = S_{/l}$ and write $T = \gamma_l(S)$ and $T_i = 
\gamma_{l+i}(S)$. Since $l$ is a secondary root, it follows from Theorem 
15 in \cite{ELG08}, that $G$ is an extension of $T/T_n$ by $R$ for $n = 
x-l$. Hence there exists a cocycle class $\gamma$ in $H^2(R, T/T_n)$ 
defining $G$. 

Note that $x \geq l_0 + e_0 \geq l+e$ and $x \equiv l+e \bmod d$. Thus 
there exists an $i$ with $x = l+e + id$ and hence $n = e+id$. 
Since $l$ is a secondary root and $e$ is an offset for it, Theorem 
\ref{isom} applies and $H^2(R, T/T_n) \cong H^2(R,T) \oplus H^3(R,T_e)$.
Let $(\delta, \beta)$ be the image of $\gamma$ under this isomorphism.

Note that $n = x-l \geq l_0-l+2e_0 \geq 2e_0$. Choosing $i = l_0+e_0$ we
obtain that $n$ and $i$ satisfy the conditions of Theorem 26 in \cite{ELG08}.
As $G$ is a descendant of $S_{/i}$, it follows from Theorem 26b) in 
\cite{ELG08} that $\delta$ defines $S$ as extension of $T$ by $R$. Further,
Theorem 26a) of \cite{ELG08} asserts that $\delta$ is equivalent to $\alpha$; 
that is,
there exists an automorphism which maps $\delta$ to $\alpha$ without
changing the isomorphism type of $G$. Hence w.l.o.g. we can replace 
$\delta$ by $\alpha$. This yields that $G$ is a group in the infinite 
sequence defined by $(l,e)$ and the cocycle class $\beta \in H^3(R, T_e)
\cong M_0$.
\end{proof}

Our construction of infinite sequences is at least partially redundant, as 
we summarize in the following theorem. 

\begin{theorem}
Let $S$ be an infinite pro-$p$-group of coclass $r$ and $k$ an integer.
\begin{items}
\item[\rm a)]
Every infinite sequence in $\T_k(S)$ is associated with a unique pair 
$(l,e) \in \P$ and a (not necessarily unique) element $\beta \in 
H^3(S_{/l}, \gamma_{l+e}(S))$. 
\item[\rm b)]
Let $\beta$ and $\beta'$ be two elements of $H^3(S_{/l}, \gamma_{l+e}(S))$
defining the infinite sequences $(G_i(\beta) \mid i \in \N_0)$ and 
$(G_i(\beta') \mid i \in \N_0)$. Then $G_i(\beta) \cong G_i(\beta')$ holds
for all $i \in \N_0$ if and only if $G_0(\beta) \cong G_0(\beta')$ holds.
\end{items}
\end{theorem}

\begin{proof}
a) The uniqueness of the pair $(l,e)$ follows from the definition of $\P$ 
and the fact that the orders of the groups in an infinite sequence depend
on $(l,e)$ modulo $d$. \\
b) This follows from Theorem 25 in \cite{ELG08} by using the definition of
offset.
\end{proof}

Our construction of infinite sequences underpins the periodicity graph
isomorphism as described in Section \ref{period}. More precisely, if
$S$, $k$ and $\P(S,k)$ are given, then with $l_0 = \max \{ l \mid (l,e) 
\in \P \}$ and $e_0 = \max \{e \mid (l,e) \in \P\}$ and $j = l_0 + 2e_0$
a graph isomorphism is induced via the infinite sequences.

\section{Parametrised presentations}
\label{secpara}

It is proved in \cite{ELG08} that the groups in an infinite sequence can 
be defined by a single parametrised presentation. We exhibit here that 
these presentations have a particularly nice form if the offset is 
divisible by the dimension of the underlying pro-$p$-group. 

Let $S$ be an infinite pro-$p$-group of dimension $d$ and let $(l,e)$
be a pair of secondary root and offset.
Let $R = S_{/l} = \langle g_1, \ldots, g_n \mid r_1, \ldots, r_m \rangle$
and let $t_1, \ldots, t_d$ be a (topological) generating set for $T = 
\gamma_l(S)$. To shorten notation, we write $t^{\v}$ for $t_1^{v_1} \cdots 
t_d^{v_d}$ where $\v = (v_1, \ldots, v_d) \in \Z_p^d$. We express the
action of $R$ on $T$ by vectors $\o_{hj} \in \Z_p^d$ with $t_j^{g_h} = 
t^{\o_{hj}}$ for $1 \leq h \leq n$ and $1 \leq j \leq d$. 

The pro-$p$-group group $S$ is an extension of $T$ by $R$ and thus it has
a pro-$p$-presentation which exhibits this extension structure. This has 
the generators $g_1, \ldots, g_n, t_1, \ldots t_d$ and its relations have 
the following form for some vectors $\v_1, \ldots, \v_m \in \Z_p^d$:
\begin{eqnarray*}
& & r_j = t^{\v_j} \mbox{ for } 1 \leq j \leq m, \\
& & t_j^{g_h} = t^{\o_{hj}} \mbox{ for } 1 \leq j \leq d, 1 \leq h \leq n,\\
& & [t_j,t_h] = 1 \mbox{ for } 1 \leq j,h \leq d.
\end{eqnarray*}

Recall that we assume that $e$ is divisible by $d$. Then the 
pro-$p$-presentation for $S$ can be modified readily to a presentation
for the finite $p$-group $S_{/l+e+id}$ by adding the relations 
$t_j^{p^{e/d+i}}$ for $1 \leq j \leq d$. The following theorem 
shows how the presentation can be modified to a presentation for a group 
in an infinite sequence.

\begin{theorem}
\label{parapres}
Let $(G_i \mid i \in \N_0)$ be an infinite sequence associated with $S$ and
$(l,e)$. Then there exist vectors $\w_1, \ldots, \w_m \in \Z_p^d$ so that 
every group $G_i$ can be defined by a presentation on the generators $g_1, 
\ldots, g_n, t_1, \ldots, t_d$ with relations of the form
\begin{eqnarray*}
& & r_j = t^{\v_j} t^{p^i \w_j} \mbox{ for } 1 \leq j \leq m, \\
& & t_j^{g_h} = t^{\o_{hj}} \mbox{ for } 1 \leq j \leq d, 1 \leq h \leq n,\\
& & [t_j,t_h] = 1 \mbox{ for } 1 \leq j,h \leq d, \\
& & t_j^{p^{e/d+i}} = 1 \mbox{ for } 1 \leq j \leq d.
\end{eqnarray*}
\end{theorem}

\begin{proof}
This follows directly from the construction on the infinite sequences: these
are extensions of $A_i = T/\gamma_{l+e+id}(S) = T/T^{p^{t+i}}$ by $R$ and the 
the considered cocycle classes $\delta_i$ have the form $\alpha_i + 
\pi_i(\beta)$, where $\pi_i$ is obtained by powering with $p^i$. The 
vectors $\v_1, \ldots, \v_m$ correspond to the cocycle class $\alpha_i$
and the vectors $\w_1, \ldots, \w_m$ correspond to $\pi_i(\beta)$.
\end{proof}

We note that the exponents in the presentation for the groups $G_i$ are
defined as elements in $\Z_p$, but the relations $t_j^{p^{e/d+i}}$ 
allow us to consider them as elements in $\Z_p / p^{e/d+i} \Z_p$.

\begin{remark}
\label{vector}
Every infinite sequence associated with $S$ and $(l,e)$ can thus be defined
by a presentation of $S$ as extension of $T$ by $R$ and a list of vectors
$\w_1, \ldots, \w_m \in \Z_p^d$.
\end{remark}

\section{The graphs $\G(2,1)$ and $\G(3,1)$}

The graphs $\G(2,1)$ and $\G(3,1)$ have been known for a long while. The
groups in these graphs have been classified by Blackburn \cite{Bla58}.
We exhibit the graphs in the following picture.

{\footnotesize
\begin{center}
\begin{tabular}{ccc}
\begindc{\undigraph}[40]
\obj(0,4){$2^2$}[\west]
\obj(0,4)[V4]{ $V_4$}
\obj(2,4){ $C4$}
\obj(0,3){ $2^3$}[\west]
\obj(0,3)[D8]{$D_8$}[\east] \mor{V4}{D8}{}
\obj(1,3)[Q8]{$Q_8$}[\east] \mor{V4}{Q8}{}
\obj(0,2){$2^4$}[\west]
\obj(0,2)[D16]{\textcolor{black}{$D_{16}$}}[\east] \mor{D8}{D16}{}
\obj(1,2)[Q16]{\textcolor{black}{$Q_{16}$}}[\east] \mor{D8}{Q16}{}
\obj(2,2)[S16]{\textcolor{black}{$S_{16}$}}[\east] \mor{D8}{S16}{}
\obj(0,1){$2^4$}[\west]
\obj(0,1)[D32]{\textcolor{black}{$D_{32}$}}[\east] \mor{D16}{D32}{}
\obj(1,1)[Q32]{\textcolor{black}{$Q_{32}$}}[\east] \mor{D16}{Q32}{}
\obj(2,1)[S32]{\textcolor{black}{$S_{32}$}}[\east] \mor{D16}{S32}{}
\obj(0,0){$2^5$}[\west]
\obj(0,0)[D64]{\textcolor{black}{$D_{64}$}}[\east] \mor{D32}{D64}{}
\obj(0,0){$\vdots$}[\south]
\obj(1,0)[Q64]{\textcolor{black}{$Q_{64}$}}[\east] \mor{D32}{Q64}{}
\obj(2,0)[S64]{\textcolor{black}{$S_{64}$}}[\east] \mor{D32}{S64}{}

\obj(0,2){}
\obj(0,1){}
\obj(0,0){}
\obj(1,2){}
\obj(1,1){}
\obj(1,0){}
\obj(2,2){}
\obj(2,1){}
\obj(2,0){}

\enddc
&\hspace{2cm} &
\begindc{\undigraph}[25]
\obj(0,6)[a1]{$3^2$}[\west]
\obj(2,6){}
\obj(0,5)[a2]{$3^3$}[\west] \mor{a1}{a2}{}
\obj(1,5)[b2]{} \mor{a1}{b2}{}
\obj(0,4)[a3]{$3^4$}[\west] \mor{a2}{a3}{}
\obj(1,4)[b3]{} \mor{a2}{b3}{}
\obj(2,4)[c3]{} \mor{a2}{c3}{}
\obj(3,4)[d3]{} \mor{a2}{d3}{}
\obj(0,3)[a4]{$3^5$}[\west] \mor{a3}{a4}{}
\obj(1,3)[b4]{} \mor{a3}{b4}{}
\obj(2,3)[c4]{} \mor{a3}{c4}{}
\obj(3,3)[d4]{} \mor{a3}{d4}{}
\obj(4,3)[e4]{} \mor{a3}{e4}{}
\obj(5,3)[f4]{} \mor{a3}{f4}{}
\obj(0,2)[a5]{$3^6$}[\west] \mor{a4}{a5}{}
\obj(1,2)[b5]{} \mor{a4}{b5}{}
\obj(2,2)[c5]{} \mor{a4}{c5}{}
\obj(3,2)[d5]{} \mor{a4}{d5}{}
\obj(4,2)[e5]{} \mor{a4}{e5}{}
\obj(5,2)[f5]{} \mor{a4}{f5}{}
\obj(6,2)[g5]{} \mor{a4}{g5}{}
\obj(0,1)[a6]{$3^7$}[\west] \mor{a5}{a6}{}
\obj(1,1)[b6]{} \mor{a5}{b6}{}
\obj(2,1)[c6]{} \mor{a5}{c6}{}
\obj(3,1)[d6]{} \mor{a5}{d6}{}
\obj(4,1)[e6]{} \mor{a5}{e6}{}
\obj(5,1)[f6]{} \mor{a5}{f6}{}
\obj(0,0)[a7]{$3^8$}[\west] \mor{a6}{a7}{}
\obj(0,0){$\vdots$}[\south] 
\obj(1,0)[b7]{} \mor{a6}{b7}{}
\obj(2,0)[c7]{} \mor{a6}{c7}{}
\obj(3,0)[d7]{} \mor{a6}{d7}{}
\obj(4,0)[e7]{} \mor{a6}{e7}{}
\obj(5,0)[f7]{} \mor{a6}{f7}{}
\obj(6,0)[g7]{} \mor{a6}{g7}{}
\enddc \\ \\
{\normalsize The graph $\G(2,1)$.} & & {\normalsize The graph $\G(3,1)$.}
\end{tabular}
\end{center} }

We demonstrate a general method to explicitly compute the pairs $\P(S,k)$ 
as in Section \ref{secpairs} and explicitly determine the cohomology groups 
as in Section \ref{secinfseq} in the example $\G(2,1)$. We then show how the
infinite sequences can be described in a highly compact form using the
presentations of Section \ref{secpara} in the example $\G(3,1)$.

Both graphs $\G(2,1)$ and $\G(3,1)$ contain a single coclass tree and this
has depth $1$. Hence we use $k=1$ in all example computations of this section. 

\subsection{The graph $\G(2,1)$}

Let $S$ be the infinite pro-$2$-group of coclass $1$. Then $S = 
\Z_2 \split C_2$, where the cyclic group $C_2$ acts by additive inversion 
on the $2$-adic integers $\Z_2$. Thus $S$ has dimension $d=1$. Further, 
its primary root is $u=2$, as $\gamma_2(S) = 2 \Z_2\cong \Z_2$ is 
torsion-free abelian and $S_{/2} \cong C_2 \times C_2$ has coclass $1$. 
\medskip

{\bf Secondary root.} Every $l \geq 2$ is a secondary root for $S$ and $k$. 
This can be observed by an explicit investigation of all groups in the graph 
using the classification by Blackburn \cite{Bla58}. Note that Theorem 
\ref{secroot} yields that every $l \geq 4$ is a possible secondary root. 
\medskip

{\bf Offset.} 
We determine the offsets for $l=2$. Let $R = S_{/2}$ and $T = \gamma_2(S)$. 
The kernel $H$ of the action of $R$ on $T$ is cyclic of order $2$. Thus we 
obtain $a = 1$ and $b = 0$ for the logarithmic exponents of the commutator 
quotient and Schur multiplicator of $H$. It follows that every $e \geq 4$ 
is an offset for $l=2$. Theorem \ref{offset} yields that every $e \geq 10$ 
is an offset for $l=2$.
\medskip

{\bf Pairs.} We choose $\P = \{(2,4)\}$ and note that this satisfies the
conditions of Definition \ref{pairs}. Thus we consider $l = 2$ and $e = 4$
and, accordingly, $R = S_{/2}$ and $T \cong \Z_2$ in the following. This
yields that $A_i = T/T_{e+id} = T/T_{4+i}$ is cyclic of order $2^{4+i}$. 
\medskip

{\bf Cohomology.}
We determine the cohomology group $H^2(R,A_i)$ explicitly with the method 
of Section 8.7.2 of \cite{HEO05}. For this purpose we use that
every element $\epsilon$ in $Z^2(R, A_i)$ defines an extension of $A_i$ by 
$R$. This extension has a presentation on the generators $g_1, g_2, t$ with 
relations 
\[ g_1^2 = t^x, \; 
   g_2^2 = t^y, \;
   g_2^{g_1} = g_2 t^z, \; 
   t^{g_1} = t^{-1}, \; 
   t^{g_2} = t, \;
   t^{2^{4+i}} = 1.\]
for certain $x,y,z \in \{0, \ldots, 2^{4+i}-1\}$. Thus we obtain a map
\[ Z^2(R, A_i) \ra (\Z_2 / 2^{4+i} \Z_2)^3 : \epsilon \ms (x,y,z). \]
Let $\hat{Z}_i$ and $\hat{B}_i$ denote the images of $Z^2(R,A_i)$ and 
$B^2(R, A_i)$, respectively, under this map. An explicit calculation 
(see Section 8.7.2 of \cite{HEO05} for details) yields that
\[ \hat{Z}_i = \langle (2^{3+i}, 0, 0), (0,0,2^{3+i}), (0,1,-1) \rangle 
    \;\; \mbox{ and } \;\; 
   \hat{B}_i = 2 \hat{Z}_i.\]

By Lemma 8.47 in \cite{HEO05}, we obtain that $H^2(R,A_i) \cong 
\hat{Z}_i/\hat{B}_i$ and thus $H^2(R, A_i)$ is elementary abelian 
of order $2^3$ for every $i \in \N_0$. 
\medskip

{\bf The image of $H^2(R,T)$.}
A presentation for $S$ as extension of $T$ by $R$ is exhibited by the 
following presentation on the generators $g_1, g_2, t$ with relations
\[ g_1^2 = 1, \; 
   g_2^2 = t, \;
   g_2^{g_1} = g_2 t^{-1}, \; 
   t^{g_1} = t^{-1}, \; 
   t^{g_2} = t. \]

Comparing this presentation with the presentation for the extensions of
$A_i$ by $R$ above shows that $\hat{\alpha}_i = (0,1,-1)$ defines the group 
$S_{/6+i}$ as extension of $A_i$ by $R$. 
\medskip

{\bf The image of $H^2(R, B_i)$.}
We find that $B_i = T_3/T_{4+i} \leq A_i$ and thus $B_i = 2^3 A_i$. Thus 
the image $\hat{M}_i$ of $H^2(R,B_i)$ in our construction for $H^2(R,A_i)$ 
can be read off easily as
\[ \hat{M}_i = 2^{3+i} \langle (1,0,0), (0,0,1) \rangle + 
               \hat{B}_i/\hat{B}_i. \]

Hence $\hat{M}_i$ is elementary abelian of order $4$ for every $i \in \N_0$.
Recall that $M_i$ is obtained from $M_0$ by powering with $p^i$. As we use
additive notation, it follows that $\hat{M}_i$ is obtained from $\hat{M}_0$
by multiplication with $2^i$.
\medskip

{\bf The infinite sequences.}
There are $4 = |M_0|$ infinite sequences arising from our construction in
this case. These are exhibited in the following table which lists for each
sequence its defining element $\hat{\beta} \in \hat{M}_0 \cong H^3(R,T_e)$,
the elements $\hat{\delta}_i = \hat{\alpha}_i + \hat{\beta}_i$ defining the
groups $G_i(\beta)$ and the names of the obtained groups. The latter shows
that there are two different cocycles which yield isomorphic sequences.

\begin{center}
\begin{tabular}{l|l|l}
\hspace{0.5cm} $\hat{\beta}$ & \hspace{1.5cm} $\hat{\delta}_i$ & name \\
\hline 
 $(0,0,0)$     & $(0,1-1)$ & dihedral groups \\
 $(2^3,0,0)$   & $(0,1-1)+2^i(2^{3},0,0)$ & quaternion groups \\
 $(2^3,0,2^3)$ & $(0,1-1)+2^i(2^{3},0,2^{3})$ & semidihedral groups \\
 $(0,0,2^3)$   & $(0,1-1)+2^i(0, 0, 2^{3})$ & semidihedral groups \\
\end{tabular}
\end{center}

\subsection{The graph $\G(3,1)$}

Let $S$ be the infinite pro-$3$-group of coclass $1$. Then $S = \Z_3^2
\split C_3$ and hence $S$ has dimension $d=2$. The classification by
Blackburn shows that $u=2$ is the primary root for $S$ and $l=3$ is
the minimal secondary root. We choose
\[ \P = \{(3,12), (4,14)\} \]
as set of pairs. This choice yields that $d \mid e$ for both offsets $e$
and hence Theorem \ref{parapres} can be used to describe the infinite
sequences corresponding to these pairs. This is done in the following.
\medskip

{\bf The case \boldmath $(l,e) = (3,12)$.}
A presentation of $S$ as extension of $\gamma_3(S)$ by $S_{/3}$ has the
generators $g_1, g_2, g_3, t_1, t_2$ and the relations
\begin{center}
\begin{tabular}{L{14.5cm}}
  g_1^3 = 1, \;
  g_2^3 = t_1 t_2, \;
  g_2^{g_1} = g_2g_3, \;
  g_3^3 = t_1^{-3} t_2^{-2}, \;
  g_3^{g_1} = g_3 t_1^2 t_2, \;
  g_3^{g_2} = g_3, \\
  t_1^{g_1} = t_1t_2, \; 
  t_1^{g_2} = t_1, \; 
  t_1^{g_3} = t_1, \;
  t_2^{g_1} = t_1^{-3}t_2^{-2}, \; 
  t_2^{g_2} = t_2, \; 
  t_2^{g_3} = t_2, \;
  t_2^{t_1} = t_2.
\end{tabular}
\end{center}
The first 3 generators and the first 6 relations in this presentation of
$S$ correspond to a presentation of $R$. By Remark \ref{vector}, each 
infinite sequence associated with this case can thus be described by a 
list of vectors $\W = (\w_1, \ldots, \w_6)$ with $\w_i \in \Z_3^2$.
Let 
\begin{center}
\begin{tabular}{R{0.6cm}C{0.5cm}L{12cm}}
  \W_1 &=& 3^5 ((0,1),(0,0),(0,0),(0,0),(0,0),(0,0)), \\
  \W_2 &=& 3^5 ((0,0),(0,0),(0,0),(0,0),(0,1),(0,0)), \\
  \W_3 &=& 3^5 ((0,0),(0,0),(0,0),(0,0),(0,0),(0,1)).
\end{tabular}
\end{center}
Then the lists of vectors $\W$ defining the infinite sequences are exactly 
the $\Z_3$-linear combinations of $\W_1, \W_2, \W_3$. As $e/d = 6$, we can
work modulo $3^6 \Z_3$. Thus there are 27 linear combinations of lists of 
vectors. Among these, the elements in $\{0, \W_1, \W_2, \W_3, \W_2+\W_3, 
\W_1+2\W_3 \}$ yield the six non-isomorphic infinite sequences.
\medskip

{\bf The case \boldmath $(l,e) = (4,14)$.}
A presentation of $S$ as extension of $\gamma_4(S)$ by $S_{/4}$ has the
generators $g_1, g_2, g_3, g_4, t_1, t_2$ and the relations
\begin{center}
\begin{tabular}{L{14.5cm}}
    g_1^3 = 1, \;
    g_2^3 = g_4^2, \;
    g_2^{g_1} = g_2g_3, \;
    g_3^3 = t_1^2 t_2^2, \;
    g_3^{g_1} = g_3g_4 t_2, \; 
    g_3^{g_2} = g_3 , \;
    g_4^3 = t_1^{-2}t_2^{-3}, \;
    g_4^{g_1} = g_4t_1t_2, \; 
    g_4^{g_2} = g_4, \;
    g_4^{g_3} = g_4, \;
    t_1^{g_1} = t_1t_2^3, \; 
    t_1^{g_2} = t_1, \; 
    t_1^{g_3} = t_1, \; 
    t_1^{g_4} = t_1, \;
    t_2^{g_1} = t_1^{-1}t_2^{-2}, \; 
    t_2^{g_2} = t_2, \; 
    t_2^{g_3} = t_2, \;
    t_2^{g_4} = t_2, \;
    t_2^{t_1} = t_2, 
\end{tabular}
\end{center}
Let $\ol{W} = \langle \W_1, \W_2, \W_3 \rangle_{\Z_3}$ with 
\begin{center}
\begin{tabular}{R{0.6cm}C{0.5cm}L{12cm}}
  \W_1 &=& 3^6 ((1,0),(0,0),(0,0),(0,0),(0,0),(0,0),(0,0),(0,0),(0,0),(0,0)), \\
  \W_2 &=& 3^6 ((0,0),(0,0),(0,0),(0,0),(0,0),(1,0),(0,0),(0,0),(0,0),(0,0)), \\
  \W_3 &=& 3^6 ((0,0),(0,0),(0,0),(1,0),(0,0),(-1,0),(0,0),(2,0),(0,0),(0,0)).
\end{tabular}
\end{center}
Then every $\W \in \ol{W}$ defines an infinite sequence. As $e/d = 7$, we 
can work modulo $3^7\Z_3$ which yields 27 elements in $\ol{W}$. The elements 
in $\{0, \W_1, \W_2, \W_3, \W_2+\W_3, \W_1+2\W_3, 2\W_2+2\W_3 \}$ yield the 
seven different infinite sequences.

\section{The graph $\G(2,2)$}

The $2$-groups of coclass $2$ have first been investigated by Newman \& 
O'Brien \cite{NOB99}. We recall the explicit picture of the resulting graph
$\G(2,2)$ in the following.

{\footnotesize
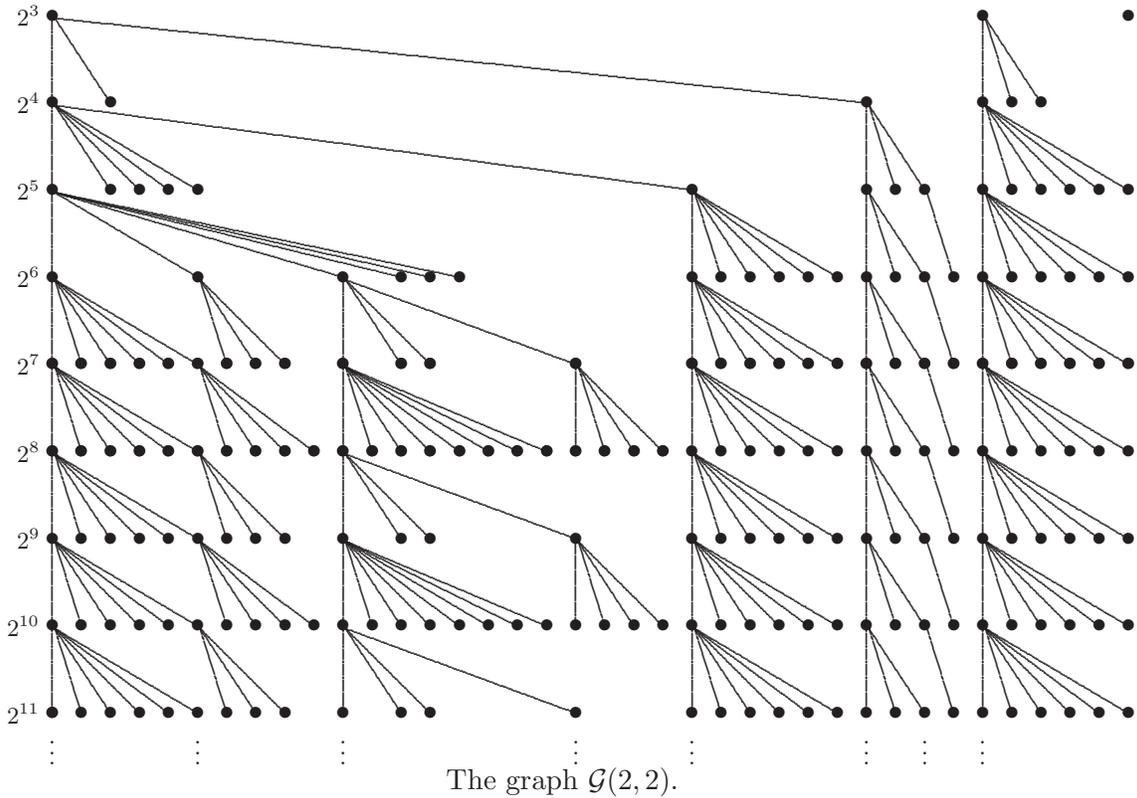
\begin{center}
\begindc{\undigraph}[11]
\obj(0,24)[a31]{$2^3$}[\west]
\obj(32,24)[a32]{}
\obj(37,24)[a33]{}
\obj(0,21)[a41]{$2^4$}[\west] \mor{a41}{a31}{}
\obj(2,21)[a43]{} \mor{a43}{a31}{}
\obj(28,21)[a42]{} \mor{a42}{a31}{}
\obj(32,21)[a44]{} \mor{a44}{a32}{}
\obj(33,21)[a45]{} \mor{a45}{a32}{}
\obj(34,21)[a46]{} \mor{a46}{a32}{}
\obj(0,18)[a51]{$2^5$}[\west] \mor{a51}{a41}{}
\obj(2,18)[a53]{} \mor{a53}{a41}{}
\obj(3,18)[a54]{} \mor{a54}{a41}{}
\obj(4,18)[a55]{} \mor{a55}{a41}{}
\obj(5,18)[a56]{} \mor{a56}{a41}{}
\obj(22,18)[a52]{} \mor{a52}{a41}{}
\obj(28,18)[a57]{} \mor{a57}{a42}{}
\obj(29,18)[a59]{} \mor{a59}{a42}{}
\obj(30,18)[a58]{} \mor{a58}{a42}{}
\obj(32,18)[a510]{} \mor{a510}{a44}{}
\obj(33,18)[a511]{} \mor{a511}{a44}{}
\obj(34,18)[a512]{} \mor{a512}{a44}{}
\obj(35,18)[a513]{} \mor{a513}{a44}{}
\obj(36,18)[a514]{} \mor{a514}{a44}{}
\obj(37,18)[a515]{} \mor{a515}{a44}{}
\obj(0,15)[a61]{$2^6$}[\west] \mor{a61}{a51}{}
\obj(5,15)[a63]{} \mor{a63}{a51}{}
\obj(10,15)[a62]{} \mor{a62}{a51}{}
\obj(12,15)[a64]{} \mor{a64}{a51}{}
\obj(13,15)[a65]{} \mor{a65}{a51}{}
\obj(14,15)[a66]{} \mor{a66}{a51}{}
\obj(22,15)[a67]{} \mor{a67}{a52}{}
\obj(23,15)[a68]{} \mor{a68}{a52}{}
\obj(24,15)[a69]{} \mor{a69}{a52}{}
\obj(25,15)[a610]{} \mor{a610}{a52}{}
\obj(26,15)[a611]{} \mor{a611}{a52}{}
\obj(27,15)[a612]{} \mor{a612}{a52}{}
\obj(28,15)[a613]{} \mor{a613}{a57}{}
\obj(29,15)[a615]{} \mor{a615}{a57}{}
\obj(30,15)[a614]{} \mor{a614}{a57}{}
\obj(31,15)[a616]{} \mor{a616}{a58}{}
\obj(32,15)[a617]{} \mor{a617}{a510}{}
\obj(33,15)[a618]{} \mor{a618}{a510}{}
\obj(34,15)[a619]{} \mor{a619}{a510}{}
\obj(35,15)[a620]{} \mor{a620}{a510}{}
\obj(36,15)[a621]{} \mor{a621}{a510}{}
\obj(37,15)[a622]{} \mor{a622}{a510}{}
\obj(0,12)[a71]{$2^7$}[\west] \mor{a71}{a61}{}
\obj(1,12)[a73]{} \mor{a73}{a61}{}
\obj(2,12)[a74]{} \mor{a74}{a61}{}
\obj(3,12)[a75]{} \mor{a75}{a61}{}
\obj(4,12)[a76]{} \mor{a76}{a61}{}
\obj(5,12)[a72]{} \mor{a72}{a61}{}
\obj(6,12)[a711]{} \mor{a711}{a63}{}
\obj(7,12)[a712]{} \mor{a712}{a63}{}
\obj(8,12)[a713]{} \mor{a713}{a63}{}
\obj(10,12)[a77]{} \mor{a77}{a62}{}
\obj(12,12)[a79]{} \mor{a79}{a62}{}
\obj(13,12)[a710]{} \mor{a710}{a62}{}
\obj(18,12)[a78]{} \mor{a78}{a62}{}
\obj(22,12)[a714]{} \mor{a714}{a67}{}
\obj(23,12)[a715]{} \mor{a715}{a67}{}
\obj(24,12)[a716]{} \mor{a716}{a67}{}
\obj(25,12)[a717]{} \mor{a717}{a67}{}
\obj(26,12)[a718]{} \mor{a718}{a67}{}
\obj(27,12)[a719]{} \mor{a719}{a67}{}
\obj(28,12)[a720]{} \mor{a720}{a613}{}
\obj(29,12)[a722]{} \mor{a722}{a613}{}
\obj(30,12)[a721]{} \mor{a721}{a613}{}
\obj(31,12)[a723]{} \mor{a723}{a614}{}
\obj(32,12)[a724]{} \mor{a724}{a617}{}
\obj(33,12)[a725]{} \mor{a725}{a617}{} 
\obj(34,12)[a726]{} \mor{a726}{a617}{}
\obj(35,12)[a727]{} \mor{a727}{a617}{}
\obj(36,12)[a728]{} \mor{a728}{a617}{}
\obj(37,12)[a729]{} \mor{a729}{a617}{}
\obj(0,9)[a81]{$2^8$}[\west] \mor{a81}{a71}{}
\obj(1,9)[a83]{} \mor{a83}{a71}{}
\obj(2,9)[a84]{} \mor{a84}{a71}{}
\obj(3,9)[a85]{} \mor{a85}{a71}{}
\obj(4,9)[a86]{} \mor{a86}{a71}{}
\obj(5,9)[a82]{} \mor{a82}{a71}{}
\obj(6,9)[a87]{} \mor{a87}{a72}{}
\obj(7,9)[a88]{} \mor{a88}{a72}{}
\obj(8,9)[a89]{} \mor{a89}{a72}{}
\obj(9,9)[a810]{} \mor{a810}{a72}{}
\obj(10,9)[a811]{} \mor{a811}{a77}{}
\obj(11,9)[a812]{} \mor{a812}{a77}{}
\obj(12,9)[a813]{} \mor{a813}{a77}{}
\obj(13,9)[a814]{} \mor{a814}{a77}{}
\obj(14,9)[a815]{} \mor{a815}{a77}{}
\obj(15,9)[a816]{} \mor{a816}{a77}{}
\obj(16,9)[a817]{} \mor{a817}{a77}{}
\obj(17,9)[a818]{} \mor{a818}{a77}{}
\obj(18,9)[a819]{} \mor{a819}{a78}{}
\obj(19,9)[a820]{} \mor{a820}{a78}{}
\obj(20,9)[a821]{} \mor{a821}{a78}{}
\obj(21,9)[a822]{} \mor{a822}{a78}{}
\obj(22,9)[a823]{} \mor{a823}{a714}{}
\obj(23,9)[a824]{} \mor{a824}{a714}{}
\obj(24,9)[a825]{} \mor{a825}{a714}{}
\obj(25,9)[a826]{} \mor{a826}{a714}{}
\obj(26,9)[a827]{} \mor{a827}{a714}{}
\obj(27,9)[a828]{} \mor{a828}{a714}{}
\obj(28,9)[a829]{} \mor{a829}{a720}{}
\obj(29,9)[a831]{} \mor{a831}{a720}{}
\obj(30,9)[a830]{} \mor{a830}{a720}{}
\obj(31,9)[a832]{} \mor{a832}{a721}{}
\obj(32,9)[a833]{} \mor{a833}{a724}{}
\obj(33,9)[a834]{} \mor{a834}{a724}{}
\obj(34,9)[a835]{} \mor{a835}{a724}{}
\obj(35,9)[a836]{} \mor{a836}{a724}{} 
\obj(36,9)[a837]{} \mor{a837}{a724}{}
\obj(37,9)[a838]{} \mor{a838}{a724}{}
\obj(0,6)[a91]{$2^9$}[\west] \mor{a91}{a81}{}
\obj(1,6)[a93]{} \mor{a93}{a81}{}
\obj(2,6)[a94]{} \mor{a94}{a81}{}
\obj(3,6)[a95]{} \mor{a95}{a81}{}
\obj(4,6)[a96]{} \mor{a96}{a81}{}
\obj(5,6)[a92]{} \mor{a92}{a81}{}
\obj(6,6)[a97]{} \mor{a97}{a82}{}
\obj(7,6)[a98]{} \mor{a98}{a82}{}
\obj(8,6)[a99]{} \mor{a99}{a82}{}
\obj(10,6)[a910]{} \mor{a910}{a811}{}
\obj(12,6)[a912]{} \mor{a912}{a811}{}
\obj(13,6)[a913]{} \mor{a913}{a811}{}
\obj(18,6)[a911]{} \mor{a911}{a811}{}
\obj(22,6)[a914]{} \mor{a914}{a823}{}
\obj(23,6)[a915]{} \mor{a915}{a823}{}
\obj(24,6)[a916]{} \mor{a916}{a823}{}
\obj(25,6)[a917]{} \mor{a917}{a823}{}
\obj(26,6)[a918]{} \mor{a918}{a823}{}
\obj(27,6)[a919]{} \mor{a919}{a823}{}
\obj(28,6)[a920]{} \mor{a920}{a829}{}
\obj(29,6)[a922]{} \mor{a922}{a829}{}
\obj(30,6)[a921]{} \mor{a921}{a829}{}
\obj(31,6)[a923]{} \mor{a923}{a830}{}
\obj(32,6)[a924]{} \mor{a924}{a833}{}
\obj(33,6)[a925]{} \mor{a925}{a833}{}
\obj(34,6)[a926]{} \mor{a926}{a833}{}
\obj(35,6)[a927]{} \mor{a927}{a833}{}
\obj(36,6)[a928]{} \mor{a928}{a833}{}
\obj(37,6)[a929]{} \mor{a929}{a833}{}
\obj(0,3)[a101]{$2^{10}$}[\west] \mor{a101}{a91}{}
\obj(1,3)[a103]{} \mor{a103}{a91}{}
\obj(2,3)[a104]{} \mor{a104}{a91}{}
\obj(3,3)[a105]{} \mor{a105}{a91}{}
\obj(4,3)[a106]{} \mor{a106}{a91}{}
\obj(5,3)[a102]{} \mor{a102}{a91}{}
\obj(6,3)[a107]{} \mor{a107}{a92}{}
\obj(7,3)[a108]{} \mor{a108}{a92}{}
\obj(8,3)[a109]{} \mor{a109}{a92}{}
\obj(9,3)[a1010]{} \mor{a1010}{a92}{}
\obj(10,3)[a1011]{} \mor{a1011}{a910}{}
\obj(11,3)[a1012]{} \mor{a1012}{a910}{}
\obj(12,3)[a1013]{} \mor{a1013}{a910}{}
\obj(13,3)[a1014]{} \mor{a1014}{a910}{}
\obj(14,3)[a1015]{} \mor{a1015}{a910}{}
\obj(15,3)[a1016]{} \mor{a1016}{a910}{}
\obj(16,3)[a1017]{} \mor{a1017}{a910}{}
\obj(17,3)[a1018]{} \mor{a1018}{a910}{}
\obj(18,3)[a1019]{} \mor{a1019}{a911}{}
\obj(19,3)[a1020]{} \mor{a1020}{a911}{}
\obj(20,3)[a1021]{} \mor{a1021}{a911}{}
\obj(21,3)[a1022]{} \mor{a1022}{a911}{}
\obj(22,3)[a1023]{} \mor{a1023}{a914}{}
\obj(23,3)[a1024]{} \mor{a1024}{a914}{}
\obj(24,3)[a1025]{} \mor{a1025}{a914}{}
\obj(25,3)[a1026]{} \mor{a1026}{a914}{}
\obj(26,3)[a1027]{} \mor{a1027}{a914}{}
\obj(27,3)[a1028]{} \mor{a1028}{a914}{}
\obj(28,3)[a1029]{} \mor{a1029}{a920}{}
\obj(29,3)[a1031]{} \mor{a1031}{a920}{}
\obj(30,3)[a1030]{} \mor{a1030}{a920}{}
\obj(31,3)[a1032]{} \mor{a1032}{a921}{}
\obj(32,3)[a1033]{} \mor{a1033}{a924}{}
\obj(33,3)[a1034]{} \mor{a1034}{a924}{}
\obj(34,3)[a1035]{} \mor{a1035}{a924}{}
\obj(35,3)[a1036]{} \mor{a1036}{a924}{}
\obj(36,3)[a1037]{} \mor{a1037}{a924}{}
\obj(37,3)[a1038]{} \mor{a1038}{a924}{}
\obj(0,0)[a111]{$2^{11}$}[\west] \mor{a111}{a101}{}
\obj(0,0){$\vdots$}[\south]
\obj(1,0)[a113]{} \mor{a113}{a101}{}
\obj(2,0)[a114]{} \mor{a114}{a101}{}
\obj(3,0)[a115]{} \mor{a115}{a101}{}
\obj(4,0)[a116]{} \mor{a116}{a101}{}
\obj(5,0)[a112]{$\vdots$}[\south] \mor{a112}{a101}{}
\obj(6,0)[a117]{} \mor{a117}{a102}{}
\obj(7,0)[a118]{} \mor{a118}{a102}{}
\obj(8,0)[a119]{} \mor{a119}{a102}{}
\obj(10,0)[a1110]{$\vdots$}[\south] \mor{a1110}{a1011}{}
\obj(12,0)[a1112]{} \mor{a1112}{a1011}{}
\obj(13,0)[a1113]{} \mor{a1113}{a1011}{}
\obj(18,0)[a1111]{$\vdots$}[\south] \mor{a1111}{a1011}{}
\obj(22,0)[a1114]{$\vdots$}[\south] \mor{a1114}{a1023}{}
\obj(23,0)[a1115]{} \mor{a1115}{a1023}{}
\obj(24,0)[a1116]{} \mor{a1116}{a1023}{}
\obj(25,0)[a1117]{} \mor{a1117}{a1023}{}
\obj(26,0)[a1118]{} \mor{a1118}{a1023}{}
\obj(27,0)[a1119]{} \mor{a1119}{a1023}{}
\obj(28,0)[a1120]{$\vdots$}[\south] \mor{a1120}{a1029}{}
\obj(29,0)[a1122]{} \mor{a1122}{a1029}{}
\obj(30,0)[a1121]{$\vdots$}[\south] \mor{a1121}{a1029}{}
\obj(31,0)[a1123]{} \mor{a1123}{a1030}{}
\obj(32,0)[a1124]{$\vdots$}[\south] \mor{a1124}{a1033}{}
\obj(33,0)[a1125]{} \mor{a1125}{a1033}{}
\obj(34,0)[a1126]{} \mor{a1126}{a1033}{}
\obj(35,0)[a1127]{} \mor{a1127}{a1033}{}
\obj(36,0)[a1128]{} \mor{a1128}{a1033}{}
\obj(37,0)[a1129]{} \mor{a1129}{a1033}{}

\tiny
\obj(0,12){}
\obj(1,12){}
\obj(2,12){}
\obj(3,12){}
\obj(4,12){}
\obj(5,12){}
\obj(6,12){}
\obj(7,12){}
\obj(8,12){}
\obj(10,12){}
\obj(12,12){}
\obj(13,12){}
\obj(18,12){}
\obj(22,12){}
\obj(23,12){}
\obj(24,12){}
\obj(25,12){}
\obj(26,12){}
\obj(27,12){}
\obj(28,12){}
\obj(29,12){}
\obj(30,12){}
\obj(31,12){}
\obj(32,12){}
\obj(33,12){}
\obj(34,12){}
\obj(35,12){}
\obj(36,12){}
\obj(37,12){}
\obj(0,9){}
\obj(1,9){}
\obj(2,9){}
\obj(3,9){}
\obj(4,9){}
\obj(5,9){}
\obj(6,9){}
\obj(7,9){}
\obj(8,9){}
\obj(9,9){}
\obj(10,9){}
\obj(11,9){}
\obj(12,9){}
\obj(13,9){}
\obj(14,9){}
\obj(15,9){}
\obj(16,9){}
\obj(17,9){}
\obj(18,9){}
\obj(19,9){}
\obj(20,9){}
\obj(21,9){}

\enddc
{\normalsize The graph $\G(2,2)$.}
\end{center} }

There are essentially five different infinite paths in $\G(2,2)$ which
correspond to the five infinite pro-$2$-groups of coclass $2$. In the
following table we list some parameters for each of these five groups.

\begin{center}
\begin{tabular}{c|cccc}
group & dimension $d$ & depth $k$ & primary root $u$ & minimal secondary 
root $l$  \\
\hline
$S_1$ & 2 & 2 & 5 & 5  \\
$S_2$ & 2 & 2 & 5 & 5  \\
$S_3$ & 1 & 1 & 4 & 4  \\
$S_4$ & 1 & 2 & 3 & 3  \\
$S_5$ & 1 & 1 & 2 & 2  \\
\end{tabular}
\end{center}

We exhibit presentations for each of the infinite sequences of $\G(2,2)$
in the same format as for $\G(3,1)$. The following table lists the pairs
we consider and gives an overview on the cohomology and the infinite
sequences in each case.

\begin{center}
\begin{tabular}{cc|ccc}
group & $(l,e)$ & $H^2(R,A_i)$ & $|H^3(R,T_e)|$ & \# non-isom sequences \\
\hline
$S_1$ & $(5,12)$ & $C_2^4 \times C_4^2$ & 16 & 10 \\
$S_1$ & $(6,14)$ & $C_2^4 \times C_4 \times C_8$ & 16 & 9 \\
$S_2$ & $(5,12)$ & $C_2^6 \times C_4$ & 32 & 12 \\
$S_2$ & $(6,14)$ & $C_2^3 \times C_8$ & 8 & 4 \\
$S_3$ & $(4,6)$ & $C_2^3 \times C_8$ & 8 & 6 \\
$S_4$ & $(3,8)$ & $C_2 \times C_4^2$ & 8 & 4 \\
$S_5$ & $(2,6)$ & $C_2^6$ & 16 & 6 
\end{tabular}
\end{center}

We use Remark \ref{vector} to describe parametrised presentations for each
infinite sequence. In each case, we outline the relevant presentations for 
$S$ and exhibit a set of lists of vectors $\W = (\w_1, \ldots, \w_m)$ with
$\w_i \in \Z_2^d$ defining the infinite sequences. Note that $\w_i$ 
corresponds to the $i$th relation of $S$ in all cases. To shorten 
notation, we sometimes replace a $0$-vector $\w_i$ by a single dot and
we eliminate $\w_j, \ldots, \w_m$ if these vectors are all $0$-vectors.

\medskip
{\bf \boldmath $S_1$ with $(5,12)$.} Generators $g_1, \ldots, g_6, t_1, t_2$ 
and relations
\begin{center}
\begin{tabular}{L{14.5cm}}
  g_1^{2} = g_4,\;
  g_2^{2} = 1, \;
  g_2^{g_1} = g_2g_3, \;
  g_3^{2} = g_6, \;
  g_3^{g_1} = g_3g_5,\;
  g_3^{g_2} = g_3g_6t_2, \;
  g_4^2 = 1,\;
  g_4^{g_1} = g_4, \; 
  g_4^{g_2} = g_4g_5g_6t_1^{-1}, \;
  g_4^{g_3} = g_4g_6, \;
  g_5^2 = t_1t_2, \; 
  g_5^{g_1} = g_5g_6t_1^{-1}, \;
  g_5^{g_2} = g_5t_1^{-1}t_2^{-1}, \;
  g_5^{g_3} = g_5, \;
  g_5^{g_4} = g_5t_1^{-1}t_2^{-1}, \;
  g_6^2 = t_2^{-1},\; 
  g_6^{g_1} = g_6t_1t_2, \; 
  g_6^{g_2} = g_6t_2,\;
  g_6^{g_3} = g_6, \;
  g_6^{g_4} = g_6t_2, \;
  g_6^{g_5} = g_6,\;
  t_1^{g_1} = t_1t_2, \; 
  t_1^{g_2} = t_1^{-1}, \; 
  t_1^{g_3} = t_1, \; 
  t_1^{g_4} = t_1^{-1}, \; 
  t_1^{g_5} = t_1, \; 
  t_1^{g_6} = t_1, \;
  t_2^{g_1} = t_1^{-2}t_2^{-1}, \; 
  t_2^{g_2} = t_2^{-1}, \; 
  t_2^{g_3} = t_2, \; 
  t_2^{g_4} = t_2^{-1}, \; 
  t_2^{g_5} = t_2, \; 
  t_2^{g_6} = t_2, 
  t_2^{t_1} = t_2.
\end{tabular}
\end{center}

{\tiny
{\centering
\begin{tabular}{L{0.4cm}C{0.2cm}L{12cm}}
\W_1 &=& 2^{5}
(\cdot,(1,0),\cdot,\cdot,(0,0),(0,1),(0,0),\cdot,(0,0),(0,0),\cdot,(0,0)), \\
\W_2 &=& 2^{5}
(\cdot,(0,1),\cdot,\cdot,(0,0),(0,0),(0,0),\cdot,(0,0),(0,0),\cdot,(0,0)), \\
\W_3 &=& 2^{5}
(\cdot,(0,0),\cdot,\cdot,(0,1),(0,0),(0,0),\cdot,(0,1),(0,1),\cdot,(0,1)), \\
\W_4 &=& 2^{5}
(\cdot,(0,0),\cdot,\cdot,(0,0),(0,0),(0,1),\cdot,(0,0),(0,0),\cdot,(0,0)), \\
\W &\in&
\{0, \W_1, \W_2, \W_3, \W_4, \W_1+\W_3, \W_1+\W_4, \W_2+\W_3, \W_2+\W_4, 
\W_3+\W_4\}
\end{tabular}}}

\medskip
{\bf \boldmath $S_1$ with $(6,14)$.} Generators $g_1, \cdots, g_6, t_1, t_2$ 
and relations
\begin{center}
\begin{tabular}{L{14.5cm}}
  g_1^{2} = g_4,\;
  g_2^{2} = 1,\;
  g_2^{g_1} = g_2g_3,\;
  g_3^{2} = g_5^2g_6, \;
  g_3^{g_1} = g_3g_5,\;
  g_3^{g_2} = g_3g_5^2g_6t_1^{-1}t_2^{-1}, \;
  g_4^2 = 1,\;
  g_4^{g_1} = g_4, \;
  g_4^{g_2} = g_4g_5g_6t_1^{-1}t_2^{-1}, \; 
  g_4^{g_3} = g_4g_5^2g_6, \;
  g_5^4 = t_1^{-1}, \;
  g_5^{g_1} = g_5g_6t_1^{-1} t_2^{-1}, \; 
  g_5^{g_2} = g_5^3t_1, \;
  g_5^{g_3} = g_5, \;
  g_5^{g_4} = g_5^3t_1, \;
  g_6^2 = t_1^2t_2, \;
  g_6^{g_1} = g_5^2g_6 t_2,\;
  g_6^{g_2} = g_6t_1^{-2} t_2^{-1},\;
  g_6^{g_3} = g_6, \;
  g_6^{g_4} = g_6t_1^{-2}t_2^{-1}, \; 
  g_6^{g_5} = g_6,\;
  t_1^{g_1} = t_1t_2^2, \; 
  t_1^{g_2} = t_1^{-1}, \; 
  t_1^{g_3} = t_1, \; 
  t_1^{g_4} = t_1^{-1}, \; 
  t_1^{g_5} = t_1, \; 
  t_1^{g_6} = t_1, \;
  t_2^{g_1} = t_1^{-1}t_2^{-1}, \; 
  t_2^{g_2} = t_2^{-1}, \; 
  t_2^{g_3} = t_2, \; 
  t_2^{g_4} = t_2^{-1}, \; 
  t_2^{g_5} = t_2, \; 
  t_2^{g_6} = t_2, \;
  t_2^{t_1} = t_2.
\end{tabular}
\end{center}

{\tiny
{\centering
\begin{tabular}{L{0.4cm}C{0.2cm}L{12cm}}
\W_1 &=& 2^{6}
(\cdot,(1,0),\cdot,\cdot,(0,0),(0,0),(0,0),\cdot,(0,0),(0,0),\cdot,(0,0)), \\
\W_2 &=& 2^{6}
(\cdot,(0,1),\cdot,\cdot,(0,0),(1,0),(0,0),\cdot,(0,0),(0,0),\cdot,(0,0)), \\
\W_3 &=& 2^{6}
(\cdot,(0,0),\cdot,\cdot,(1,0),(0,0),(0,0),\cdot,(1,0),(1,0),\cdot,(1,0)), \\
\W_4 &=& 2^{6}
(\cdot,(0,0),\cdot,\cdot,(0,0),(0,0),(1,0),\cdot,(0,0),(0,0),\cdot,(0,0)), \\
\W &\in& \{0, \W_1, \W_2, \W_3, \W_4, \W_1+\W_3, \W_1+\W_4, \W_2+\W_4, \W_2+\W_3\}.
\end{tabular} } }
\medskip

{\bf \boldmath $S_2$ with $(5,12)$.} Generators $g_1, \ldots, g_6, t_1, t_2$ 
and relations
\begin{center}
\begin{tabular}{L{14.5cm}}
  g_1^{2} = g_4, \;
  g_2^{2} = 1, \;
  g_2^{g_1} = g_2g_3, \;
  g_3^{2} = 1, \;
  g_3^{g_1} = g_3g_5, \;
  g_3^{g_2} = g_3, \;
  g_4^2 = 1, \;
  g_4^{g_1} = g_4, \;
  g_4^{g_2} = g_4g_5t_1^{-1}, \;
  g_4^{g_3} = g_4g_6t_1^{-1}t_2^{-1}, \;
  g_5^2 = g_6t_1, \;
  g_5^{g_1} = g_5g_6, \;
  g_5^{g_2} = g_5t_1^{-1}, \;
  g_5^{g_3} = g_5t_1^{-1}, \;
  g_5^{g_4} = g_5t_1^{-1}, \;
  g_6^2 = t_2, \;
  g_6^{g_1} = g_6t_1^{-1} t_2^{-1}, \; 
  g_6^{g_2} = g_6t_1, \;
  g_6^{g_3} = g_6t_2^{-1}, \; 
  g_6^{g_4} = g_6t_2^{-1}, \;
  g_6^{g_5} = g_6, \;
  t_1^{g_1} = t_1t_2, \; 
  t_1^{g_2} = t_1^{-1}, \; 
  t_1^{g_3} = t_1^{-1}, \; 
  t_1^{g_4} = t_1^{-1}, \; 
  t_1^{g_5} = t_1, \; 
  t_1^{g_6} = t_1, \;
  t_2^{g_1} = t_1^{-2}t_2^{-1}, \; 
  t_2^{g_2} = t_1^2t_2, \; 
  t_2^{g_3} = t_2^{-1}, \; 
  t_2^{g_4} = t_2^{-1}, \; 
  t_2^{g_5} = t_2, \; 
  t_2^{g_6} = t_2, \;
  t_2^{t_1} = t_2.
\end{tabular}
\end{center}

{\tiny
{\centering
\begin{tabular}{L{0.4cm}C{0.2cm}L{12.9cm}}
\W_1 &=& 2^{5}
(\cdot,(0,1),\cdot,(0,0),\cdot,(0,0),(0,0),\cdot,(0,0),(0,0),(0,0),(0,0),(0,0),(0,0),(0,0),(0,0),(0,0),(0,0),(0,0),(0,0)), \\
\W_2 &=& 2^{5}
(\cdot,(0,0),\cdot,(1,0),\cdot,(1,0),(0,0),\cdot,(0,0),(0,0),(1,0),(1,1),(1,0),(1,1),(1,0),(0,1),(1,0),(1,1),(0,1),(0,1)), \\ 
\W_3 &=& 2^{5}
(\cdot,(0,0),\cdot,(0,1),\cdot,(0,1),(0,0),\cdot,(0,0),(0,0),(0,1),(0,1),(0,1),(0,1),(0,1),(0,0),(0,1),(0,1),(0,0),(0,0)), \\ 
\W_4 &=& 2^{5}
(\cdot,(0,0),\cdot,(0,0),\cdot,(0,0),(0,1),\cdot,(0,0),(0,0),(0,0),(0,0),(0,0),(0,0),(0,0),(0,0),(0,0),(0,0),(0,0),(0,0)), \\ 
\W_5 &=& 2^{5}
(\cdot,(0,0),\cdot,(0,0),\cdot,(0,0),(0,0),\cdot,(0,1),(0,1),(0,1),(0,0),(0,1),(0,1),(0,1),(0,0),(0,1),(0,1),(0,0),(0,0)), \\
\W &\in&
\{0, \W_1, \W_2, \W_3, \W_4, \W_1+\W_4, \W_1+\W_3, \W_1+\W_2, \W_1+\W_3+\W_4, \W_1+\W_2+\W_4,
\W_3+\W_4, \W_2+\W_4 \}.
\end{tabular} } }

\medskip
{\bf \boldmath $S_2$ with $(6,14)$.} Generators $g_1, \ldots, g_7, t_1, t_2$ 
and relations
\begin{center}
\begin{tabular}{L{14.5cm}}
  g_1^{2} = g_4, \;
  g_2^{2} = 1, \;
  g_2^{g_1} = g_2g_3, \;
  g_3^{2} = 1, \;
  g_3^{g_1} = g_3g_5, \;
  g_3^{g_2} = g_3, \;
  g_4^2 = 1, \;
  g_4^{g_1} = g_4, \;
  g_4^{g_2} = g_4g_5g_7, \;
  g_4^{g_3} = g_4g_6g_7t_2, \;
  g_5^2 = g_7t_1^{-1}, \;
  g_5^{g_1} = g_5g_6, \;
  g_5^{g_2} = g_5g_7, \;
  g_5^{g_3} = g_5g_7, \;
  g_5^{g_4} = g_5g_7, \;
  g_6^2 = g_8t_2^{-1}, \;
  g_6^{g_1} = g_6g_7 t_2, \;
  g_6^{g_2} = g_6g_7t_1^{-1}, \;
  g_6^{g_3} = g_6t_2, \;
  g_6^{g_4} = g_6t_2, \;
  g_6^{g_5} = g_6, \;
  g_7^2 = t_1, \;
  g_7^{g_1} = g_7t_2, \;
  g_7^{g_2} = g_7t_1^{-1}, \;
  g_7^{g_3} = g_7t_1^{-1}, \;
  g_7^{g_4} = g_7 t_1^{-1}, \;
  g_7^{g_5} = g_7, \;
  g_7^{g_6} = g_7, \;
  t_1^{g_1} = t_1t_2^2, \; 
  t_1^{g_2} = t_1^{-1}, \; 
  t_1^{g_3} = t_1^{-1}, \; 
  t_1^{g_4} = t_1^{-1}, \; 
  t_1^{g_5} = t_1, \; 
  t_1^{g_6} = t_1, \;
  t_2^{g_1} = t_1^{-1}t_2^{-1}, \; 
  t_2^{g_2} = t_1t_2, \; 
  t_2^{g_3} = t_2^{-1}, \; 
  t_2^{g_4} = t_2^{-1}, \; 
  t_2^{g_5} = t_2, \; 
  t_2^{g_6} = t_2, \;
  t_2^{t_1} = t_2.
\end{tabular}
\end{center}

{\tiny
{\centering
\begin{tabular}{L{0.4cm}C{0.2cm}L{13cm}}
\W_1 &=& 2^{6}
(\cdot,\cdot,\cdot,(1,0),\cdot,(1,0),(0,0),\cdot,\cdot,(0,0),(1,0),(1,0),(1,0),(1,0),(1,0),(0,0),(1,0),(1,0),(0,0),(0,0),\cdot,\cdot,(0,0)), \\
\W_2 &=& 2^{6}
(\cdot,\cdot,\cdot,(0,0),\cdot,(0,0),(1,0),\cdot,\cdot,(0,0),(0,0),(0,0),(0,0),(0,0),(0,0),(0,0),(0,0),(0,0),(0,0),(0,0),\cdot,\cdot,(0,0)), \\
\W_3 &=& 2^{6}
(\cdot,\cdot,\cdot,(0,0),\cdot,(0,0),(0,0),\cdot,\cdot,(1,0),(0,0),(0,0),(0,0),(0,0),(0,0),(1,0),(1,0),(0,0),(1,0),(1,0),\cdot,\cdot,(1,0)), \\
\W &\in& \{0, \W_1, \W_2, \W_1+\W_2 \}.
\end{tabular} } }

\medskip
{\bf \boldmath $S_3$ with $(4,6)$.} Generators $g_1, \ldots, g_5, t$ and 
relations
\begin{center}
\begin{tabular}{L{14.5cm}}
  g_1^{2} = g_4, \;
  g_2^{2} = 1, \;
  g_2^{g_1} = g_2g_3, \;
  g_3^{2} = g_5, \;
  g_3^{g_1} = g_3g_5t^{-1}, \;  
  g_3^{g_2} = g_3g_5t^{-1}, \;
  g_4^2 = 1, \;
  g_4^{g_1} = g_4, \;
  g_4^{g_2} = g_4, \;
  g_4^{g_3} = g_4, \;
  g_5^2 = t, \;
  g_5^{g_1} = g_5t^{-1}, \;
  g_5^{g_2} = g_5t^{-1}, \;  
  g_5^{g_3} = g_5, \;
  g_5^{g_4} = g_5, \;
  t^{g_1} = t^{-1}, \; 
  t^{g_2} = t^{-1}, \; 
  t^{g_3} = t, \; 
  t^{g_4} = t, \; 
  t^{g_5} = t.
\end{tabular}
\end{center}

{\tiny
{\centering
\begin{tabular}{L{0.4cm}C{0.2cm}L{12cm}}
\W_1 &=& 2^{5} ((0),(1),(0),(0),(0),(0),(0),(0),(0),(0),(0),(0),(0),(0),(0)), \\
\W_2 &=& 2^{5} ((0),(0),(0),(0),(0),(1),(0),(0),(1),(0),(1),(1),(1),(0),(0)), \\
\W_3 &=& 2^{5} ((0),(0),(0),(0),(0),(0),(1),(0),(0),(0),(0),(0),(0),(0),(0)), \\
\W &\in& \{0, \W_1, \W_2, \W_3, \W_1+\W_2+\W_3, \W_2+\W_3 \}.
\end{tabular} } }
\medskip

{\bf \boldmath $S_4$ with $(3,8)$.} Generators $g_1, \ldots, g_4, t$ and 
relations
\begin{center}
\begin{tabular}{L{14.5cm}}
  g_1^{2} = g_4, \;
  g_2^{2} = g_3, \;
  g_2^{g_1} = g_2g_3t^{-1}, \;
  g_3^{2} = t, \;
  g_3^{g_1} = g_3t^{-1}, \;  
  g_3^{g_2} = g_3, \;
  g_4^2 = 1, \;
  g_4^{g_1} = g_4, \;
  g_4^{g_2} = g_4, \;
  g_4^{g_3} = g_4, \;
  t^{g_1} = t^{-1}, \;  
  t^{g_2} = t, \;  
  t^{g_3} = t, \;  
  t^{g_4} = t. 
\end{tabular}
\end{center}

{\tiny
{\centering
\begin{tabular}{L{0.4cm}C{0.2cm}L{12cm}}
\W_1 &=& 2^{7} ((0),(0),(0),(0),(0),(0),(1),(0,(0),(0)), \\ 
\W_2 &=& 2^{6} ((0),(0),(0),(1),(1),(0),(-2),(0),(2),(0)), \\
\W &\in& \{0, \W_1, 3\W_2, 2\W_2 \}.
\end{tabular} } }

\medskip
{\bf \boldmath $S_5$ with $(2,6)$.} Generators $g_1, \ldots, g_3, t$ and 
relations
\begin{center}
\begin{tabular}{L{14.5cm}}
  g_1^{2} = 1, \;
  g_2^{2} = t, \;
  g_2^{g_1} = g_2t^{-1}, \;
  g_3^{2} = 1, \;
  g_3^{g_1} = g_3, \;
  g_3^{g_2} = g_3, \;
  t^{g_1} = t^{-1}, \;  
  t^{g_2} = t, \;  
  t^{g_3} = t. 
\end{tabular}
\end{center}

{\tiny 
{\centering 
\begin{tabular}{L{0.4cm}C{0.2cm}L{12cm}}
\W_1 &=& 2^{5} ((1),(0),(0),(0),(0),(0)), \\
\W_2 &=& 2^{5} ((0),(0),(1),(0),(0),(0)), \\
\W_3 &=& 2^{5} ((0),(0),(0),(0),(1),(0)), \\
\W_4 &=& 2^{5} ((0),(0),(0),(0),(0),(1)), \\
\W &\in& \{0, \W_1, \W_2, \W_3, \W_4, \W_1+\W_4\}.
\end{tabular} } }

\bibliographystyle{abbrv}
\bibliography{paper-ccfams}

\begin{thebibliography}{10}

\bibitem{Bla58}
N.~Blackburn.
\newblock On a special class of $p$-groups.
\newblock {\em Acta Math.}, 100:45--92, 1958.

\bibitem{DEF08}
H.~Dietrich, B.~Eick, and D.~Feichtenschlager.
\newblock Investigating $p$-groups by coclass with {\sc gap}.
\newblock In {\em {C}omputational group theory and the theory of groups},
  volume 470 of {\em Contemp. Math.}, pages 45--61. Amer. Math. Soc.,
  Providence, RI, 2008.

\bibitem{DuS01}
M.~du~Sautoy.
\newblock Counting $p$-groups and nilpotent groups.
\newblock {\em Inst. Hautes Etudes Sci. Publ. Math.}, 92:63--112, 2001.

\bibitem{Eic06b}
B.~Eick.
\newblock Automorphism groups of 2-groups.
\newblock {\em J. Algebra}, 300:91--101, 2006.

\bibitem{Eic07b}
B.~Eick.
\newblock Schur multiplicator of finite $p$-groups with fixed coclass.
\newblock {\em Israel J. Math.}, 166:157--166, 2008.

\bibitem{EFe10b}
B.~Eick and D.~Feichtenschlager.
\newblock Computation of low-dimensional (co)homology groups for infinite
  sequences of p-groups with fixed coclass.
\newblock {\em In preparation}.

\bibitem{ELG08}
B.~Eick and C.~R. Leedham-Green.
\newblock On the classification of prime-power groups by coclass.
\newblock {\em Bull. Lond. Math. Soc.}, 40(2), 2008.

\bibitem{HEO05}
D.~Holt, B.~Eick, and E.~A. O'Brien.
\newblock {\em Handbook of computional group theory}.
\newblock Chapman \& Hall, 2005.

\bibitem{LGM02}
C.~R. Leedham-Green and S.~McKay.
\newblock {\em The structure of groups of prime power order}.
\newblock London Mathematical Society Monographs New Series. Oxford Science
  Publication, 2002.

\bibitem{LNe80}
C.~R. Leedham-Green and M.~F. Newman.
\newblock Space groups and groups of prime-power order {I}.
\newblock {\em Archiv der Mathematik}, 35:193--202, 1980.

\bibitem{NOB99}
M.~F. Newman and E.~A. O'Brien.
\newblock Classifying 2-groups by coclass.
\newblock {\em Trans. Amer. Math. Soc.}, 351:131--169, 1999.

\end{thebibliography}

\end{document}